\documentclass[12pt]{article}

\usepackage{latexsym,amssymb,upref,amsmath,amsthm,amsfonts}
\usepackage {color}
\usepackage {pstricks, pst-node}
\usepackage {graphicx}
\usepackage {latexsym}
\usepackage [T1]{fontenc}
\usepackage [cp1250]{inputenc}
\usepackage{bbm}
\usepackage{bbold}

\newcommand{\achr}{\mathrm{achr}}
\newcommand{\ro}{\mathbb R}
\newcommand{\co}{\mathbb C}
\newcommand{\x}{\mathbb X}
\newcommand{\bp}{\noindent\textit{Proof. }}
\newcommand{\ep}{\blacksquare}

\newcommand{\exc}{\mathrm{exc}}
\newcommand{\frq}{\mathrm{frq}}
\newcommand{\cov}{\mathrm{cov}}
\newcommand{\Cov}{\mathrm{Cov}}
\newcommand{\Col}{\mathrm{Col}}

\newcommand{\wt}{\mathrm{wt}}
\newcommand{\cM}{\mathcal{M}}

\newtheorem{theorem}{Theorem}
\newtheorem{lemma}[theorem]{Lemma}
\newtheorem{proposition}[theorem]{Proposition}

\newtheorem{claim}{Claim}
\newtheorem{corollary}[theorem]{Corollary}

\title{\bf The achromatic number of the Cartesian product of $K_6$ and $K_q$}

\author{Mirko Hor\v n\' ak\thanks{\noindent E-mail address: mirko.hornak$@$upjs.sk}\\
Institute of Mathematics, P.J. \v{S}af\'arik University\\
Jesenn\'a 5, 040\ 01 Ko\v{s}ice, Slovakia}

\date{}

\begin{document}
\maketitle

\begin{abstract}
Let $G$ be a graph and $C$ a finite set of colours. A vertex colouring $f:V(G)\to C$ is complete if for any pair of distinct colours $c_1,c_2\in C$ one can find an edge $\{v_1,v_2\}\in E(G)$ such that $f(v_i)=c_i$, $i=1,2$. The achromatic number of $G$ is defined to be the maximum number $\achr(G)$ of colours in a proper complete vertex colouring of $G$. In the paper $\achr(K_6\square K_q)$ is determined for any integer $q$ such that either $8\le q\le40$ or $q\ge42$ is even.
\end{abstract}


\section{Introduction}

Let $G$ be a finite simple graph and $C$ a finite set of colours. A vertex colouring $f:V(G)\to C$ is \textit{complete} provided that for any pair $\{c_1,c_2\}\in\binom C2$ (of distinct colours of $C$) there exists an edge $\{v_1,v_2\}$ (usually shortened to $v_1v_2$) of $G$ such that $f(v_i)=c_i$, $i=1,2$. The \textit{achromatic number} of $G$, in symbols $\achr(G)$, is the maximum cardinality of the colour set in a proper complete vertex colouring of $G$. The achromatic number was introduced in Harary, Hedetniemi, and Prins~\cite{HaHePr}, where among other things the following interpolation result was proved:

\begin{theorem}\label{ip}
If $G$ is a graph, and an integer $k$ satisfies $\chi(G)\le k\le\achr(G)$, there exists a proper complete vertex colouring of $G$ using $k$ colours. $\qed$
\end{theorem}

In the present paper the achromatic number of $K_6\square K_q$, the Cartesian product of $K_6$ and $K_q$ (the notation following Imrich and Klav\v zar~\cite{IK} is adopted), is determined for all $q$'s satisfying either $8\le q\le40$ or $q\ge42$ and $q\equiv0\pmod2$. This is the second in a series of three papers, in which the problem of finding $\achr(K_6\square K_q)$ is completely solved. Some historical remarks concerning the achromatic number, a motivation of the problem and basic facts on proper complete colourings of Cartesian products of two complete graphs are available in the first paper \cite{Ho1}, where $\achr(K_6\square K_q)$ has been determined for odd $q\ge41$. Here we reproduce those facts from~\cite{Ho1} that are necessary for our paper. Maybe a bit surprisingly to prove that $\achr(K_6\square K_7)=18$ has required quite a long analysis contained in the third paper \cite{Ho2}.

For $m,n\in\mathbb Z$ we work with \textit{integer intervals} defined by 
\begin{equation*}
[m,n]=\{z\in\mathbb Z:m\le z\le n\},\qquad [m,\infty)=\{z\in\mathbb Z:m\le z\}.
\end{equation*}
If $p,q\in[1,\infty)$ and $V(K_r)=[1,r]$, $r=p,q$, then $V(K_p\square K_q)=[1,p]\times[1,q]$, and $E(K_p\square K_q)$ consists of edges $(i_1,j_1)(i_2,j_2)$, where $i_1,i_2\in[1,p]$ and $j_1,j_2\in[1,q]$ satisfy either $i_1=i_2$ and $j_1\ne j_2$ or $i_1\ne i_2$ and $j_1=j_2$.

Let $\mathcal M(p,q,C)$ denote the set of $p\times q$ matrices $M$ with entries from $C$ such that all lines (rows and columns) of $M$ have pairwise distinct entries, and any pair $\{\alpha,\beta\}\in\binom C2$ is \textit{good in $M$}, which means that there is a line of $M$ containing both $\alpha$ and $\beta$; the pair $\{\alpha,\beta\}$ is \textit{row-based} or \textit{column-based} (in $M$) depending on whether the involved line is a row or a column, respectively. In other words, the number of lines witnessing the fact that the pair $\{\alpha,\beta\}$ is good, is positive, and it may happen that the pair $\{\alpha,\beta\}$ is simultaneously row-based and column-based as well. For a matrix $M$ we denote by $(M)_{i,j}$ the entry of $M$ appearing in the $i$th row and the $j$th column. 
 
\begin{proposition}[see \cite{Ho1}]\label{general}
If $p,q\in[1,\infty)$ and $C$ is a finite set, then the following statements are equivalent:

$\mathrm(1)$ There is a proper complete vertex colouring of $K_p\square K_q$ using as colours elements of $C$.

$\mathrm(2)$ $\mathcal M(p,q,C)\ne\emptyset$.\quad$\qed$
\end{proposition}

The implication $(2)\Rightarrow(1)$ of Proposition~\ref{general} is based on a straightforward observation that if $M\in\mathcal M(p,q,C)$, then the vertex colouring $f_M$ of $K_p\square K_q$ defined by $f_M(i,j)=(M)_{i,j}$ is proper and complete as well.

\begin{proposition}[see \cite{Ho1}]\label{wlog}
If $p,q\in[1,\infty)$, $C,D$ are finite sets, $M\in\mathcal M(p,q,$ $C)$, mappings $\rho:[1,p]\to[1,p]$, $\sigma:[1,q]\to[1,q]$, $\pi:C\to D$ are bijections, and $M_{\rho,\sigma}$, $M_{\pi}$ are $p\times q$ matrices defined by $(M_{\rho,\sigma})_{i,j}=(M)_{\rho(i),\sigma(j)}$ and $(M_{\pi})_{i,j}=\pi((M)_{i,j})$, then $M_{\rho,\sigma}\in\mathcal M(p,q,C)$ and $M_{\pi}\in\mathcal M(p,q,D)$.\quad$\qed$
\end{proposition}

Let $M\in\mathcal M(p,q,C)$. The \textit{frequency} of a colour $\gamma\in C$ is the number $\frq(\gamma)$ of times $\gamma$ appears in $M$, while $\frq(M)$, the \textit{frequency} of $M$, is the minimum of frequencies of colours in $C$. A colour of frequency $l$ is an \textit{$l$-colour}, $C_l$ is the set of $l$-colours and $c_l=|C_l|$. $C_{l+}$ is the set of colours of frequency at least $l$ and $c_{l+}=|C_{l+}|$. For the (complete) colouring $f_M$ mentioned above denote $V_{\gamma}=f_M^{-1}(\gamma)\subseteq[1,p]\times[1,q]$, and let $N(V_{\gamma})$ be the neighbourhood of $V_{\gamma}$ (the union of neighborhoods of vertices in $V_{\gamma}$). The \textit{excess} of $\gamma$ is defined to be the maximum number $\exc(\gamma)$ of vertices in a set $S\subseteq N(V_{\gamma})$ such that the restriction of $f_M$ formed by uncolouring the vertices of $S$ is still complete with respect to pairs of colours containing $\gamma$. The \textit{excess} of $M$ is the minimum $\exc(M)$ of excesses of colours in $C$.

We denote by $\ro(i)$ the set of colours in the $i$th row of $M$ and by $\co(j)$ the set of colours in the $j$th column of $M$. Further, let
\begin{alignat*}{2}
\ro_l(i)&=C_l\cap\ro(i),\qquad &r_l(i)&=|\ro_l(i)|,\\
\co_l(j)&=C_l\cap\co(j), &c_l(j)&=|\co_l(j)|,\\
\end{alignat*}
\vskip-6mm
\noindent so that $\ro_l(i)$ and $\co_l(j)$ is the set of $l$-colours appearing in the row $i$ and those appearing in the column $j$, respectively. For $i,j,k\in[1,p]$ let
\begin{alignat*}{2}
\ro(i,j)&=C_2\cap\ro(i)\cap\ro(j), &r(i,j)&=|\ro(i,j)|,\\
\ro(i,j,k)&=C_3\cap\ro(i)\cap\ro(j)\cap\ro(k),\qquad &r(i,j,k)&=|\ro(i,j,k)|,
\end{alignat*}
\noindent and for $m,n\in[1,q]$ let
\begin{equation*}
\co(m,n)=C_2\cap\co(m)\cap\co(n),\qquad c(m,n)=|\co(m,n)|.
\end{equation*}
Thus $\ro(i,j)$ and $\co(m,n)$ is the set of 2-colours appearing in both rows $i,j$ and those appearing in both columns $m,n$, respectively, while $\ro(i,j,k)$ is the set of 3-colours appearing exactly in the rows $i,j,k$. For $\gamma\in C$ define
\begin{equation*}
\mathbb R(\gamma)=\{i\in[1,p]:\gamma\in\mathbb R(i)\}
\end{equation*}
as the set of (numbers of) rows containing the colour $\gamma$.

If $S\subseteq[1,p]\times[1,q]$, we say that a colour $\gamma\in C$ \textit{occupies a position in} $S$ (\textit{appears in $S$} or simply \textit{is in} $S$) if there exists $(i,j)\in S$ with $(M)_{i,j}=\gamma$. For a nonempty set of colours $A\subseteq C$, the \textit{set of columns covered by} $A$ is
\begin{equation*}
\Cov(A)=\{j\in[1,q]:\co(j)\cap A\ne\emptyset\},
\end{equation*}
\textit{i.e.}, the set of (numbers of) columns containing a colour of $A$. We put $\cov(A)=|\Cov(A)|$, and for $A=\{\alpha\},\{\alpha,\beta\},\{\alpha,\beta,\gamma\}$ we use a simplified notation $\Cov(\alpha)$, $\Cov(\alpha,\beta)$, $\Cov(\alpha,\beta,\gamma)$ and $\cov(\alpha)$, $\cov(\alpha,\beta)$, $\cov(\alpha,\beta,\gamma)$ instead of $\Cov(A)$ and $\cov(A)$.

\section{Matrix constructions}\label{mat}

\begin{proposition}\label{468}
$\achr(K_6\square K_4)\ge12$, $\achr(K_6\square K_6)\ge18$ and $\achr(K_6\square K_8)\ge21$.
\end{proposition}

\begin{proof}
Let $M_q$ be the $6\times q$ matrix below, $q=4,6,8$, where $\bar n$ stands for $10+n$ and $\bar{\bar n}$ for $20+n$.
\[
\begin{pmatrix}
1 &2 &3 &4\\
5 &6 &7 &8\\
9 &\bar0 &\bar1 &\bar2\\
2 &1 &4 &3\\
7 &8 &5 &6\\
\bar2 &\bar1 &\bar0 &9
\end{pmatrix}
\quad
\begin{pmatrix}
{1} &{2} &{3} &{4} &{5} &{6}\\
{7} &{8} &{9} &{\bar0} &{\bar1} &{\bar2}\\
{\bar3} &{\bar4} &{\bar5} &{\bar6} &{\bar7} &{\bar8}\\
2 &1 &\bar7 &\bar2 &\bar5 &\bar0\\
\bar1 &\bar8 &4 &3 &7 &\bar4\\
\bar6 &9 &8 &\bar3 &6 &5
\end{pmatrix}
\quad
\begin{pmatrix}
1 &2 &3 &4 &5 &\bar6 &\bar7 &\bar8\\
6 &7 &8 &9 &\bar0 &\bar8 &\bar6 &\bar7\\
\bar1 &\bar2 &\bar3 &\bar4 &\bar5 &\bar7 &\bar8 &\bar6\\
4 &8 &7 &1 &\bar9 &\bar5 &\bar{\bar0} &\bar{\bar1}\\
\bar3 &5 &\bar1 &\bar{\bar1} &2 &9 &\bar9 &\bar{\bar0}\\
\bar0 &\bar4 &\bar{\bar0} &\bar2 &6 &3 &\bar{\bar1} &\bar9
\end{pmatrix}
\]
One can easily check that $M_4\in\mathcal M(6,4,[1,12])$, $M_6\in\mathcal M(6,6,[1,18])$ and $M_8\in\mathcal M(6,8,[1,21])$. Therefore, we are done by Proposition~\ref{general}. 
\end{proof}

Let $r\in[3,9]$ and $\bar r=\lceil\frac{r-1}3\rceil$ so that $\bar r\in[1,3]$. Consider $r$-element colour sets  
\[ V_r=\{v_r(k):k\in[1,r]\},\ W_r=\{w_r(k):k\in[1,r]\} \] 
such that $V_r$, $W_r$ and $[1,18]$ are pairwise disjoint. Further, let $N_r$ be the $6\times r$ matrix with entries defined by
\begin{alignat*}{3}
(N_r)_{i,j}&=v_r(i+j-1), &i&\in[1,3], &j&\in[1,r],\\
(N_r)_{i,j}&=w_r((i-4)\bar r+j),\quad &i&\in[4,6],\quad &j&\in[1,r],
\end{alignat*} 
where $u_r(m)=u_r(n)$ for $u\in\{v,w\}$ and $m,n\in\mathbb Z$, $m\equiv n\pmod r$.

\begin{lemma}\label{39}
If $r\in[3,9]$, then $N_r\in\mathcal M(6,r,V_r\cup W_r)$.
\end{lemma} 
\begin{proof}
Let $C=V_r\cup W_r$. We first show that the colouring $f_{N_r}:[1,6]\times[1,r]\to C$ is proper. For that purpose suppose $i,i_1,i_2\in[1,6]$, $i_1\ne i_2$, $j,j_1,j_2\in[1,r]$, $j_1\ne j_2$. 

From $|i+j_1-1-(i+j_2-1)|=|j_1-j_2|\in[1,r-1]$ for $i\in[1,3]$ and $|(i-4)\bar r+j_1-[(i-4)\bar r+j_2]|=|j_1-j_2|\in[1,r-1]$ for $i\in[4,6]$ it follows that $j_1\not\equiv j_2\pmod r$, and so $(N_r)_{i,j_1}\ne(N_r)_{i,j_2}$.

If $i_1,i_2\in[4,6]$, then $|(i_1-4)\bar r+j-[(i_2-4)\bar r+j]|=|i_1-i_2|\bar r\in[1,r-1]$, because $|i_1-i_2|\in[1,2]$, and from $0<(r-1)/3$ it follows that $1\le\bar r<2\bar r\le\frac{2(r+1)}3<r$. Therefore, $(N_r)_{i_1,j}\ne(N_r)_{i_2,j}$.

If $i_1\in[1,3]$ and $i_2\in[4,6]$, the desired assertion $(N_r)_{i_1,j}\ne(N_r)_{i_2,j}$ follows from $V_r\cap W_r=\emptyset$, and if $i_1,i_2\in[1,3]$, it suffices to realise that $i_1+j-1\ne i_2+j-1$, since $|i_1+j-1-(i_2+j-1)|=|i_1-i_2|\in[1,2]$.

Further, we have to show that each pair $\{\alpha,\beta\}\in\binom C2$ is good in $N_r$. If both $\alpha,\beta$ are either in $V_r$ or in $W_r$, then $\mathbb R(\alpha)=\mathbb R(\beta)\in\{[1,3],[4,6]\}$, and the pair $\{\alpha,\beta\}$ is row-based.

If $\alpha\in V_r$ and $\beta\in W_r$, there is $j\in[1,r]$ such that $\alpha=(N_r)_{1,j}=(N_r)_{2,j-1}=(N_r)_{3,j-2}$, hence $\alpha$ shares a common column with any colour of the set
$W_{r,j}=\{w_r((i-4)\bar r+j-k):i\in[4,6],k\in[0,2]\}$. We have
\begin{equation*}
I_{r,j}=\bigcup_{i=4}^6\bigcup_{k=0}^2\{(i-4)\bar r+j-k\}=\bigcup_{i=4}^6[j-2+(i-4)\bar r,j+(i-4)\bar r]=[j-2,j+2\bar r],
\end{equation*}
because $j-2+l\bar r\le j+(l-1)\bar r$, which means that there is no gap between integer intervals $[j-2+(l-1)\bar r,j+(l-1)\bar r]$ and $[j-2+l\bar r,j+l\bar r]$, $l=1,2$. Further, since $|I_{r,j}|=3+2\bar r\ge r$ (the assumption $r-1\ge3+2\lceil\frac{r-1}3\rceil$ would lead to $r-1\ge3+2\frac{r-1}3$ and $r\ge10$, a contradiction), the integer interval $I_{r,j}$ covers all congruence classes modulo $r$. This yields $W_{r,j}=\bigcup_{m\in I_{r,j}}\{w_r(m)\}=W_r$, and so the pair $\{\alpha,\beta\}$ is column-based.
\end{proof} 

\begin{proposition}\label{915}
If $q\in[9,15]$, then $\achr(K_6\square K_q)\ge2q+6$.
\end{proposition}
\begin{proof}
The block matrix $M_q=(M_6\,N_{q-6})$ belongs to $\mathcal M(6,q,C)$ with $C=[1,18]\cup V_{q-6}\cup W_{q-6}$. To see it first realise that since the colourings $f_{M_6}$ and $f_{N_{q-6}}$ are proper (Proposition~\ref{468}, Lemma~\ref{39}), and $[1,18]\cap(V_{q-6}\cup W_{q-6})=\emptyset$, the colouring $f_{M_q}$ is proper, too.

Next, we have to show that each pair $\{\alpha,\beta\}\in\binom C2$ is good in $M_q$. The colourings $f_{M_6}$ and $f_{N_{q-6}}$ are complete, hence it suffices to restrict our attention to $\alpha\in[1,18]$ and $\beta\in V_{q-6}\cup W_{q-6}$. In such a case $|\mathbb R(\alpha)\cap[1,3]|=1=|\mathbb R(\alpha)\cap[4,6]|$ and $\mathbb R(\beta)\in\{[1,3],[4,6]\}$ so that $\mathbb R(\alpha)\cap\mathbb R(\beta)\ne\emptyset$, and the pair $\{\alpha,\beta\}$ is row-based.

So, Proposition~\ref{general} yields $\achr(K_6\square K_q)\ge|C|=18+2(q-6)=2q+6$.
\end{proof}

For $l=0,1,2,3$ let $r_l\in[3,9]$, $\bar r_l=\lceil\frac{r_l-1}3\rceil$, and let $N_{r_l}^l$ be the $6\times r_l$ matrix with 
\begin{alignat*}{3}
(N_{r_l}^l)_{i,j}&=v_{r_l}^l(i+j-1), &i&\in[1,3], &j&\in[1,r_l],\\
(N_{r_l}^l)_{i,j}&=w_{r_l}^l((i-4)\bar r_l+j),\quad &i&\in[4,6],\quad &j&\in[1,r_l],
\end{alignat*} 
where the sets $[1,12]$ and
\[ V_{r_l}^l=\{v_{r_l}^l(k):k\in[1,r_l]\},\ W_{r_l}^l=\{w_{r_l}^l(k):k\in[1,r_l]\},\ l=0,1,2,3, \]
are pairwise disjoint; moreover, we suppose that $u_{r_l}^l(m)=u_{r_l}^l(n)$ for $u\in\{v,w\}$ and $m,n\in\mathbb Z$, $m\equiv n\pmod {r_l}$. Further, let $\tilde N_{r_l}^l$ be the $6\times r_l$ matrix obtained from $N_{r_l}^l$ by interchanging its rows $l$ and $l+3$, $l=1,2,3$.

\begin{proposition}\label{1640}
If $q\in[16,40]$, then $\achr(K_6\square K_q)\ge2q+4$.
\end{proposition}
\begin{proof}
Since $4\cdot3=12\le q-4\le36=4\cdot9$, there are integers $r_l\in[3,9]$, $l=0,1,2,3$, such that $\sum_{l=0}^3 r_l=q-4$. Let us show that the block matrix $M_q=(M_4\,N_{r_0}^0\,\tilde N_{r_1}^1\,\tilde N_{r_2}^2\,\tilde N_{r_3}^3)$ belongs to $\mathcal M(6,q,C)$ with $C=[1,12]\cup\bigcup_{l=0}^3(V_{r_l}^l\cup W_{r_l}^l)$.

By Lemma~\ref{39} and Proposition~\ref{wlog} we have $N_{r_l}\in\mathcal M(6,r_l,V_{r_l}\cup W_{r_l})$ and $N_{r_l}^l\in\mathcal M(6,r_l,V_{r_l}^l\cup W_{r_l}^l)$, $l=0,1,2,3$, as well as $\tilde N_{r_l}^l\in\mathcal M(6,r_l,V_{r_l}^l\cup W_{r_l}^l)$, $l=1,2,3$. The colouring $f_M$ with $M\in\{M_4,N_{r_0}^0\,\tilde N_{r_1}^1\,\tilde N_{r_2}^2\,\tilde N_{r_3}^3\}$ is proper, and the sets $[1,12]$, $V_{r_l}^l\cup W_{r_l}^l$, $l=0,1,2,3$, are pairwise disjoint, hence the colouring $f_{M_q}$ is proper.

Now consider a pair $\{\alpha,\beta\}\in\binom C2$. If both $\alpha,\beta$ are either in $[1,12]$ or in $V_{r_l}^l\cup W_{r_l}^l$ with $l\in[0,3]$, then the pair $\{\alpha,\beta\}$ is good in $M_q$, because the colourings $f_{M_4}$ and $f_M$ with $M\in\{N_{r_0}^0\,\tilde N_{r_1}^1\,\tilde N_{r_2}^2\,\tilde N_{r_3}^3\}$ are complete
(Propositions~\ref{wlog},\,\ref{468}, Lemma~\ref{39}).

In all remaining cases we show that $\mathbb R(\alpha)\cap\mathbb R(\beta)\ne\emptyset$, which means that the pair $\{\alpha,\beta\}$ is row-based.

If $\alpha\in[1,12]$ and $\beta\in\bigcup_{l=0}^3(V_{r_l}^l\cup W_{r_l}^l)$, then $\mathbb R(\alpha)\in\{\{1,4\},\{2,5\},\{3,6\}\}$ and $\mathbb R(\beta\}\in\{\{1,2,3\},\{1,2,6\},\{1,3,5\},\{1,5,6\},\{2,3,4\},\{2,4,6\},\{3,4,5\},\{4,5,6\}\}$ so that $\mathbb R(\alpha)\cap\mathbb R(\beta)\ne\emptyset$ follows immediately.

If $\alpha\in V_{r_0}^0$ and $\beta\in\bigcup_{l=1}^3 V_{r_l}^l$, then $\mathbb R(\alpha)\cap\mathbb R(\beta)=[1,3]\setminus\{l\}$; similarly, with $\alpha\in W_{r_0}^0$ and $\beta\in\bigcup_{l=1}^3 W_{r_l}^l$ we have $\mathbb R(\alpha)\cap\mathbb R(\beta)=[4,6]\setminus\{l\}$. 

If $\{i,j,k\}=[1,3]$, $\alpha\in V_{r_i}^i$ and $\beta\in V_{r_j}^j$, then $\mathbb R(\alpha)\cap\mathbb R(\beta)=\{k\}$; the same conclusion holds provided that $\{i,j,k\}=[4,6]$, $\alpha\in W_{r_i}^i$ and $\beta\in W_{r_j}^j$.

If there is $l\in[1,3]$ such that either $\alpha\in V_{r_0}^0$ and $\beta\in W_{r_l}^l$ or $\alpha\in W_{r_0}^0$ and $\beta\in V_{r_l}^l$, then $\mathbb R(\alpha)\cap\mathbb R(\beta)=\{l\}$.

Finally, $\alpha\in V_{r_i}^i$ with $i\in[1,3]$ and $\beta\in W_{r_j}^j$ with $j\in[1,3]\setminus\{i\}$ leads to $\mathbb R(\alpha)\cap\mathbb R(\beta)=\{i+3\}$.

Thus the colouring $f_{M_q}$ is complete and $M_q\in\mathcal M(6,q,C)$. Since $|C|=12+\sum_{l=0}^32r_l=2q+4$, by Proposition~\ref{general} we get $\achr(K_6\square K_q)\ge2q+4$.
\end{proof}

For a given $s\in[2,\infty)$ consider colour sets $U_s=\{u_k:k\in[1,s]\}$ with $U\in\{X,Y,Z,T\}$ such that the sets $[1,12],X_s,Y_s,Z_s$ and $T_s$ are pairwise disjoint. 

\begin{proposition}\label{16+lb}
If $q\in[42,\infty)$ and $q\equiv0\pmod2$, then $\achr(K_6\square K_q)\ge 2q+4$.
\end{proposition}
\begin{proof}
Let $s=\frac{q-4}2$, and let $M_q$ be the $6\times q$ matrix below. We show that $M_q\in\mathcal M(6,q,C)$ for $C=[1,12]\cup X_s\cup Y_s\cup Z_s\cup T_s$. Obviously, since $s\ge19\ge2$, the colouring $f_{M_q}$ is proper.
\setcounter{MaxMatrixCols}{14}
\begin{equation*}
\begin{pmatrix}
1 &2 &3 &4 &x_1 &x_2 &\dots &x_{s-1} &x_s &y_1 &y_2 &\dots &y_{s-1} &y_s\\
5 &6 &7 &8 &x_s &x_1 &\dots &x_{s-2} &x_{s-1} &z_1 &z_2 &\dots &z_{s-1} &z_s\\
9 &10 &11 &12 &t_1 &t_2 &\dots &t_{s-1} &t_s &x_1 &x_2 &\dots &x_{s-1} &x_s\\
2 &1 &4 &3 &z_1 &z_2 &\dots &z_{s-1} &z_s &t_1 &t_2 &\dots &t_{s-1} &t_s\\
7 &8 &5 &6 &t_s &t_1 &\dots &t_{s-2} &t_{s-1} &y_s &y_1 &\dots &y_{s-2} &y_{s-1}\\
12 &11 &10 &9 &y_1 &y_2 &\dots &y_{s-1} &y_s &z_s &z_1 &\dots &z_{s-2} &z_{s-1}
\end{pmatrix}
\textbf{}\end{equation*}

Notice that $M_q$ has a submatrix $M_4$ (formed by the first four columns of $M_q$). The colouring $f_{M_4}$ is complete (Proposition~\ref{468}), hence a pair $\{\alpha,\beta\}\in\binom C2$ is good in $M_q$ provided that $\alpha,\beta\in[1,12]$. So, it remains to consider pairs $\{\alpha,\beta\}$ with $\alpha\in C$ and $\beta\in X_s\cup Y_s\cup Z_s\cup T_s$. Realise that $\ro(\alpha)\in\mathcal R_1\cup\mathcal R_2$ and $\ro(\beta)\in\mathcal R_2$, where $\mathcal R_1=\{\{1,4\},\{2,5\},\{3,6\}\}$ and $\mathcal R_2=\{\{1,2,3\},\{1,5,6\},\{2,4,6\},\{3,4,5\}\}$. As $R\cap R_2\ne\emptyset$ for any $R\in\mathcal R_1\cup\mathcal R_2$ and any $R_2\in\mathcal R_2$, the pair $\{\alpha,\beta\}$ is row-based.

Thus, by Proposition~\ref{general} we see that $\achr(K_6\square K_q)\ge 4s+12=2q+4$. \end{proof}

\section{Some basic facts concerning matrices in $\mathcal M(p,q,C)$}

In this section we first reproduce those facts from~\cite{Ho1} that are necessary for our paper.

\begin{lemma}[see \cite{Ho1}]\label{bgen}
If $p,q\in[1,\infty)$, $C$ is a finite set, $M\in\mathcal M(p,q,C)$ and $\gamma\in C$, then the following hold:

$1.$ $\frq(\gamma)\le\min(p,q)$;

$2.$ $\frq(\gamma)=l$ implies $\exc(\gamma)=l(p+q-l-1)-(|C|-1)\ge0$;

$3.$ $\frq(M)=l$ implies $|C|\le\lfloor\frac{pq}l\rfloor$. $\qed$
\end{lemma}

\begin{lemma}[see \cite{Ho1}]\label{exc}
If $p,q\in[1,\infty)$, $C$ is a finite set and $M\in\mathcal M(p,q,C)$, then $\exc(M)=\exc(\gamma)$, where $\gamma\in C$ satisfies $\frq(\gamma)=\frq(M)$. $\qed$
\end{lemma}

\begin{lemma}[see \cite{Ho1}]\label{+m}
If $q\in[7,\infty)$, $s\in[0,7]$, $C$ is a set of cardinality $2q+s$ and $M\in\mathcal M(6,q,C)$, then the following hold:

$1.$ $c_1=0$;

$2.$ $c_l=0$ for $l\in[7,\infty)$;

$3.$ $c_2\ge3s$;

$4.$ $c_{3+}\le2q-2s$;

$5.$ $\sum_{i=3}^6 ic_i\le6q-6s$;

$6.$ $\frq(M)=2$;

$7.$ $\exc(M)=7-s$;

$8.$ $c_{4+}\le c_2-3s$;

$9.$ if $\{i,k\}\in\binom{[1,6]}{2}$, then $r(i,k)\le8-s$. \qed
\end{lemma}

\begin{lemma}\label{not7} 
If $q\in[7,\infty)$, then $\achr(K_6\square K_q)\le2q+6$.
\end{lemma}
\begin{proof}
If our lemma is false, according to Theorem~\ref{ip} and Proposition~\ref{general} there is a $(2q+7)$-element set $C$ and $M\in\mathcal M(6,q,C)$. By Lemma~\ref{+m}.3 and \ref{+m}.7 then $c_2\ge3\cdot7=21$ and $\exc(M)=0$. Further, by Proposition~\ref{wlog} we may suppose without loss of generality that $r_2(1)\ge r_2(i)$ for $i\in[2,6]$, which implies $6r_2(1)\ge\sum_{i=1}^6 r_2(i)=2c_2$ and $r_2(1)\ge\lceil\frac{2c_2}6\rceil\ge7$; we shall use on similar occasions (w) to indicate that it is Proposition~\ref{wlog}, which is behind the fact that the generality is not lost. As a consequence there is $i\in[2,6]$ such that $r(1,i)\ge\lceil\frac{r_2(1)}5\rceil\ge2$. With $\gamma\in\ro(1,i)$ then each colour of $\ro(1,i)\setminus\{\gamma\}$ contributes one to the excess of $\gamma$, hence we have $0=\exc(M)=\exc(\gamma)\ge r(1,i)-1\ge1$, a contradiction.
\end{proof}

\section{Solution}\label{sol}

It turns out that the matrix constructions given by Propositions~\ref{468} and \ref{915}--\ref{1640} are optimum from the point of view of $\achr(K_6\square K_q)$. The optimality was already known for $q=4$ (Hor\v n\'ak and Puntig\'an~\cite{HoPu}) and $q=6$ (Bouchet~\cite{B}), while the rest of the present paper is devoted to the analysis of remaining $q$'s.

\begin{theorem}\label{K6K8}
$\achr(K_6\square K_8)=21$.
\end{theorem}
\begin{proof} By Proposition~\ref{468} and Lemma~\ref{not7} we know that $21\le\achr(K_6\square K_8)\le22$. Suppose that $\achr(K_6\square K_8)=22$; because of Proposition~\ref{general} there is a 22-element set of colours $C$ and a matrix $M\in\mathcal M(6,8,C)$. With $q=8$ and $s=6$ Lemma~\ref{+m} yields $c_2\ge18$, $c_{3+}\le4$, $\frq(M)=2$, $\exc(M)=1$, $c_{4+}\le c_2-18$ and $r(i,k)\le2$ for $\{i,k\}\in\binom{[1,6]}2$. We are going to strengthen step by step the requirements on $M$ to finally reach a conclusion that $M$ cannot exist at all.

\begin{claim}\label{ij}
If $i\in[1,6]$ and $j\in[1,8]$, then $|\ro_2(i)\cap\co(j)|\le2$.
\end{claim}

\bp Suppose that $|C_2'|\ge3$ for $C_2'=\ro_2(i)\cap\co(j)$. If $\alpha=(M)_{i,j}\in C_2$, then each colour of $C_2'\setminus\{\alpha\}$ contributes one to the excess of $\alpha$ so that, by Lemma~\ref{exc}, $1=\exc(M)=\exc(\alpha)\ge|C_2'\setminus\{\alpha\}|\ge2$, a contradiction. On the other hand, if $\alpha\in C_{3+}$, then for any $\beta\in C_2'$ we have $1=\exc(M)=\exc(\beta)\ge|C_2'\setminus\{\beta\}|\ge2$, a contradiction again. $\ep$

\begin{claim}\label{r22}
If $\{i,k\}\in\binom{[1,6]}2$ and $\alpha,\beta\in\ro_2(i,k)$, $\alpha\ne\beta$, then $\cov(\alpha,\beta)=2$, so that $\{\alpha,\beta\}\subseteq\co(j)$ for both $j\in\Cov(\alpha,\beta)$.
\end{claim}
 
\bp Suppose (w) $i=1$, $k=2$ and $(M)_{1,1}=(M)_{2,2}=\alpha$. If $\cov(\alpha,\beta)=4$, (w) $(M)_{1,3}=(M)_{2,4}=\beta$. Denote $A=\ro(1)\cup\ro(2)$. From $\exc(\alpha)=\exc(\beta)=1$ it is clear that $|A|=14$, $|C\setminus A|=8$, and that, for both $l\in[1,2]$, each colour of $C\setminus A$ occupies a position in the set $[3,6]\times[2l-1,2l]$ (the colouring $f_M$ is complete). Since $|(C\setminus A)\cap C_2|=|C\setminus A|-|(C\setminus A)\cap C_{3+}|\ge8-c_{3+}\ge4$, there is a colour $\gamma\in(C\setminus A)\cap C_2$. The neighbourhood of the 2-element vertex set $f_M^{-1}(\gamma)$ contains 10 vertices belonging to $[3,6]\times[1,4]$, all coloured with 7 colours of $(C\setminus A)\setminus\{\gamma\}$. Thus 3 colours of $(C\setminus A)\setminus\{\gamma\}$ appear in the above neighbourhood twice. As a consequence we obtain $1=\exc(M)=\exc(\gamma)\ge10-7=3$, a contradiction. 

If $\cov(\alpha,\beta)=3$, (w) $(M)_{1,2}=(M)_{2,3}=\beta$. With $B=\ro(1)\cup\ro(2)\cup\co(2)$ then $\exc(\alpha)=1$ implies $|B|=18$ and $|C\setminus B|=4$.
Each colour $\gamma\in C\setminus B$ belongs to $\co(1)\cap\co(3)$ (both pairs $\{\gamma,\alpha\}$ and $\{\gamma,\beta\}$ are good in $M$) and satisfies $\exc(\gamma)\ge|(C\setminus B)\setminus\{\gamma\}|=3$, hence $\gamma\notin C_2$ and $\gamma\in C_{3+}$.  Consequently, $c_{3+}\ge4$, $c_{3+}=4$, $C\setminus B=C_{3+}$, $B=C_2$, $\delta=(M)_{1,3}\in C_2$, and the second copy of $\delta$ appears in $[3,6]\times[4,8]$ so that $\exc(\delta)\ge2$ (if $\delta=(M)_{m,n}$ with $(m,n)\in[3,6]\times[4,8]$, then both $\beta$ and $(M)_{m,1}\in C_{3+}$ contribute to the excess of $\delta$), a contradiction. $\ep$
\vskip1mm

If $\Cov(\alpha,\beta)=\{j,l\}$ for the 2-colours $\alpha,\beta$ of Claim~\ref{r22}, then (w) $\alpha=(M)_{i,j}=(M)_{k,l}$ and $\beta=(M)_{i,l}=(M)_{k,j}$; we say that the set of 2-colours $\{\alpha,\beta\}$ forms an \textit{$\x$-configuration} (in $M$): both copies of a colour $\gamma\in\{\alpha,\beta\}$ are diagonal to each other in the ``rectangle'' of the matrix $M$ with corners $(i,j),(i,l),(k,j),(k,l)$, and this fact will be in the sequel for simplicity denoted by $\{\alpha,\beta\}\to\x$.

\begin{claim}\label{r26}
If $i\in[1,6]$, then $r_2(i)=6$ and $r_3(i)=2$.
\end{claim}

\bp Let (w) $r_2(1)\ge r_2(i)$ for each $i\in[2,6]$ and $r(1,i)\ge r(1,i+1)$ for each $i\in[2,5]$. Suppose that $r_2(1)\ge7$. Then $2\ge r(1,2)\ge\lceil\frac{r_2(1)}5\rceil=2$ so that $r(1,2)=2$. Moreover, $r(1,6)\le\lfloor\frac{r_2(1)}5\rfloor=1$, and there is $p\in[2,5]$ such that $r(1,p)=2$ and $r(1,p+1)<2$. As $r_2(6)=8-r_{3+}(6)\ge8-c_{3+}\ge4$, we have $|\ro_2(6)\setminus\ro_2(1)|=|\ro_2(6)\setminus\ro(1,6)|=r_2(6)-r(1,6)=(8-r_{3+}(6))-r(1,6)\ge4-1=3$, and there exists $\alpha\in\ro_2(6)\setminus\ro_2(1)$.

Consider a colour $\beta\in\ro(i,k)$, where $1<i<k$. When counting the number of pairs $\{\beta,\gamma\}$ with $\gamma\in\ro_2(1)$, that are good in $M$ because of the copy of $\beta$ in the $m$th row of $M$, $m\in\{i,k\}$, we see that $r(1,m)$ of them are row-based, and, by Claim~\ref{ij}, at most two of them are column-based.
There is $i\in[2,5]$ such that $\alpha\in\ro(i,6)$. Proceeding as above we obtain that the number of pairs $\{\alpha,\gamma\}$ with $\gamma\in\ro_2(1)$, that are good in $M$, is at most $\rho=[r(1,i)+2]+[r(1,6)+2]\le6+r(1,6)$. Observe that we cannot have $r(1,6)\le r_2(1)-7$, because then $\rho\le6+[r_2(1)-7]<r_2(1)$, a contradiction. Therefore, $1\ge r(1,6)\ge r_2(1)-6\ge1$, $r(1,6)=1$ and $r_2(1)=7$. Now $p\le3$, because $p\ge4$ would mean $7=r_2(1)=\sum_{i=2}^6r(1,i)\ge2(p-1)+(6-p)=p+4\ge8$, a contradiction. Consequently, any colour $\delta\in C_2\setminus\ro(1)$ needs a copy in a row $k\in[2,p]$ (if both copies of $\delta$ are only in rows numbered from $p+1$ to 6, the number of pairs $\{\delta,\gamma\}$, $\gamma\in\ro_2(1)$, that are good in $M$,
is at most $2(1+2)<r_2(1)$, which is impossible), hence $18\le c_2\le7+\sum_{i=2}^p(r_2(i)-r(1,i))\le7+(p-1)(7-2)=5p+2\le17$, a contradiction. 

Since the assumption $r_2(1)\ge7$ was false, we have $36\ge\sum_{i=1}^6r_2(i)=2c_2\ge36$. Therefore, $r_2(i)=6$ for each $i\in[1,6]$, $c_2=18$, $c_{4+}=0$, $C=C_2\cup C_3$, and the proof follows. $\ep$

By Claim~\ref{r26}, for every $i\in[1,6]$ there is (at least one) $p_i\in[1,6]\setminus\{i\}$ such that $r(i,p_i)=\lceil\frac65\rceil=2$. Then, by Claim~\ref{r22}, $\ro(i,p_i)\to\x$ for $i\in[1,6]$. Let $\tilde C_2=\{\alpha,\beta,\gamma,\delta\}\subseteq C_2$ be such that $\{\alpha,\beta\}\to\x$ and $\{\gamma,\delta\}\to\x$, where $\{\alpha,\beta\}\ne\{\gamma,\delta\}$ (which immediately yields $\{\alpha,\beta\}\cap\{\gamma,\delta\}=\emptyset$). We have $\cov(\tilde C_2)\in[3,4]$, since with $\cov(\tilde C_2)=2$ each of $\beta,\gamma,\delta$ contributes one to the excess of $\alpha$ so that $\exc(\alpha)\ge3$, a contradiction. 

Thus (w) $\cov(\alpha,\beta)=[1,2]$ and $\cov(\gamma,\delta)=[l,l+1]$, where $l\in[2,3]$. Then $\co(1)\cup\co(2)\subseteq C_2$ (because of $r_3(1)=r_3(2)=2$ and $\exc(\alpha)=1$, colours of $C_{3}$ occupy four positions in $\ro(1)\cup\ro(2)$ and neither position in $\co(1)\cup\co(2)$), and, similarly, $\co(l)\cup\co(l+1)\subseteq C_2$. So, all 3-colours appear exclusively in $7-l$ columns of $M$ numbered from $l+2$ to 8. There is $j\in[l+2,8]$ such that $c_3(j)\ge\lceil\frac{3c_3}{7-l}\rceil\ge\lceil\frac{12}{4}\rceil=3$, while $c_2(j)\ge6-c_3=2$. If $\varepsilon\in\co_2(j)\cap\ro(m,n)$, then 3-colours occupy at least $r_3(m)+r_3(n)+[c_3(j)-1]=c_3(j)+3\ge6$ positions in $N(V_{\varepsilon})$ (since at least $c_3(j)-1$ colours of $\co_3(j)$ are not in $\ro(m)\cup\ro(n)$), hence $\exc(\varepsilon)\ge6-c_3=2$, a final contradiction for the proof of Theorem~\ref{K6K8}.
\end{proof}

\begin{theorem}\label{915f}
If $q\in[9,15]$, then $\achr(K_6\square K_q)=2q+6$.
\end{theorem}
\begin{proof}
See Proposition~\ref{915} and Lemma~\ref{not7}.
\end{proof}

\begin{theorem}\label{16+}
If either $q\in[42,\infty)$ and $q\equiv0\pmod2$ or $q\in[16,40]$, then $\achr(K_6\square K_q)=2q+4$.
\end{theorem}

\begin{proof}
We proceed by the way of contradiction. Since $\achr(K_6\square K_q)\ge2q+4$ (Propositions~\ref{1640} and \ref{16+lb}), the assumption $\achr(K_6\square K_q)\ge2q+5$ by Theorem~\ref{ip} and Proposition~\ref{general} means that there is a  colour set $C$ with $|C|=2q+5=2q+s$ and a matrix $M\in\mathcal M(6,q,C)$. By Lemma~\ref{+m} then $C=\bigcup_{l=2}^6C_l$, $c_2\ge15$, $c_{3+}\le2q-10$, $\sum_{i=3}^{6}ic_i\le6q-30$, $\frq(M)=2=\exc(M)$, $c_{4+}\le c_2-15$, and $\{i,k\}\in\binom{[1,6]}{2}$ implies $r(i,k)\le3$. A contradiction is reached first for $q\ge19$, then for $q\in[17,18]$, and finally for $q=16$. 

Let $G$ be an auxiliary graph $G$ associated with $M$, in which $V(G)=[1,6]$ and $\{i,k\}\in E(G)$ if and only if $r(i,k)\ge1$. 

\begin{claim}\label{Dle3}
$\Delta(G)\le3$.
\end{claim}

\bp In \cite{Ho1} it has been proved that $\Delta(G)\ge4$ implies $q\le40-5\cdot5=15$, a contradiction. $\ep$
\vskip1mm

In~\cite{Ho1} one can find also proofs of the following two claims.

\begin{claim}\label{D3}
If $\{i,j,k,l,m,n\}=[1,6]$, $r(i,l)\ge1$, $r(j,l)\ge1$ and $r(k,l)\ge1$, then $r(l,m,n)\ge q-9$. $\ep$
\end{claim}

\begin{claim}\label{r3r3}
If $\{i,j,k,l,m,n\}=[1,6]$ and $r(i,j,k)\ge1$, then $r(l,m,n)\le9$. $\ep$
\end{claim}

\begin{claim}\label{add}
If $\{i,j,k,l,m,n\}=[1,6]$, $r(i,l)\ge1$, $r(j,l)\ge1$ and $r(k,l)\ge1$, $\{a,b\}\in\binom{[1,6]}{2}$ and $r(a,b)\ge1$, then $|\{i,j,k\}\cap\{a,b\}|=1=|\{l,m,n\}\cap\{a,b\}|$.
\end{claim}
\bp The assumptions of Claim~\ref{add} imply that $r(l,m,n)\ge q-9\ge7$ (see Claim~\ref{D3}), hence there is $\alpha\in\ro(a,b)$. 

If $\{a,b\}\subseteq\{i,j,k\}$, the number of pairs $\{\alpha,\beta\}$, $\beta\in\ro(l,m,n)$, that are good in $M$ (and necessarily column-based), is at most $3\,\cov(\alpha)=6<r(l,m,n)$, a contradiction. 

On the other hand, with $\{a,b\}\subseteq\{l,m,n\}$ each colour of $\ro(l,m,n)$ contributes one to the excess of $\alpha$ so that $2=\exc(M)=\exc(\alpha)\ge r(l,m,n)\ge7$, a contradiction again. 

Therefore, $2=|\{i,j,k,l,m,n\}\cap\{a,b\}|=|\{i,j,k\}\cap\{a,b\}|+|\{l,m,n\}\cap\{a,b\}|\le1+1=2$, and then $|\{i,j,k\}\cap\{a,b\}|=1=|\{l,m,n\}\cap\{a,b\}|$. $\ep$

\begin{claim}\label{r2}
If $\{i,k\}\in\binom{[1,6]}2$, then $r(i,k)\le2$.
\end{claim}
\bp Let (w) $i=1$, $k=2$, $\Cov(\ro(1,2))=[1,n]$, and assume (for a proof by contradiction) that $r(1,2)=3$ (see Lemma~\ref{+m}.9), which implies $n\in[3,6]$. 

We are going to show that $A=C\setminus(\ro(1)\cup\ro(2))\subseteq C_{3+}$. First observe that each colour $\alpha\in A$ occupies at least two positions in $S_n=[3,6]\times[1,n]$ (all pairs $\{\alpha,\beta\}$, $\beta\in\ro(1,2)$, are good in $M$).

If $n=3$, then $2q+5=|C|=|A|+|\ro(1)\cup\ro(2)|\le\lfloor\frac{4\cdot3}2\rfloor+2q-3=2q+3$, a contradiction.

If $n=4$, then $2q+5\le\lfloor\frac{4\cdot4}2\rfloor+(2q-3)=2q+5$, $|A|=8$, and any  colour of $A$ occupies exactly two positions in $S_4$. Suppose there is a colour $\alpha\in A\cap C_2$. If a vertex $(i,j)\in S_4$ belongs to $N(V_{\alpha})$, then $(M)_{i,j}\in A\setminus\{\alpha\}$, hence $2=\exc(M)=\exc(\alpha)\ge10-|N(V_{\alpha})\setminus\{\alpha\}|=3$ (the set $N(V_{\alpha})$ has 10 vertices in $S_4$), a contradiction. Therefore, $A\subseteq C_{3+}$.

If $n=5$, there is $j\in[1,5]$ such that $|\co(j)\cap\ro(1,2)|=2$ and $|\co(l)\cap\ro(1,2)|=1$ for $l\in[1,5]\setminus\{j\}$. Then $A=A_2\cup A_3\cup A_4$, where $A_l$ consists of colours of $A$ occupying $l$ positions in $S_5$. With $a_l=|A_l|$ we obtain $a_2\le4$ (if $\alpha\in A_2\setminus\co(j)$, at least one of three pairs $\{\alpha,\beta\}$ with $\beta\in\ro(1,2)$ is not good in $M$, a contradiction), $a_3+a_4=|A|-a_2\ge8-4=4$, $16+a_3+a_4\le16+a_3+2a_4\le2(a_2+a_3+a_4)+a_3+2a_4=\sum_{l=2}^4la_l=4\cdot5=20$, $a_3+a_4\le20-16=4$, $a_3+a_4=4$, all six above expressions are 20, which implies $a_4=0$, $a_3=4=a_2$, and then all positions in $S_5$ are occupied by colours of $A_2\cup A_3$. If $\Cov(\ro(1,2)\setminus\co(j))=\{s,t\}\subseteq[1,5]\setminus\{j\}$, then $A_2\subseteq\co(j)\cup\co(s)\cup\co(t)$. For the set $B$ of colours in $C_2\setminus A_2$ that are not in $[1,2]\times[1,5]$ we have $|B|\ge c_2-[(2|[1,5]|-3)+a_2]\ge15-11=4$. However, the number of pairs $\{\alpha,\beta\}$ with $\alpha\in A_2$ and $\beta\in B$, that are good in $M$, is at most three (if $\beta\in\ro(u)$, $u\in[3,6]$, only $(M)_{u,j},(M)_{u,s}$ and $(M)_{u,t}$ are available as $\alpha$), a contradiction.

If $n=6$, the frequency of each colour in $\alpha\in A$ is at least three, since all pairs $\{\alpha,\beta\}$ with $\beta\in\ro(1,2)$ are column-based, and at most one of them satisfies the implication $\alpha\in\co(j)\Rightarrow\beta\in\co(j)$.

Thus $A\subseteq C_{3+}$ and $C_2\subseteq\ro(1)\cup\ro(2)$. For $l\in[1,6]$ let
\begin{equation*}
K(l)=\{m\in[1,6]\setminus\{l\}:r(l,m)\ge1\}
\end{equation*}
so that $\Delta(G)\le3$ (Claim~\ref{Dle3}) implies $|K(l)|\le3$.
With $r_2^-(m,3-m)=|\ro_2(m)\setminus\ro_2(3-m)|$, $m=1,2$, let $p\in[1,2]$ be such that $r_2^-(p,3-p)\ge r_2^-(3-p,p)$. Observe that $3-p\in K(p)$ and $r_2^-(p,3-p)=\sum_{l\in K(p)\setminus\{3-p\}}r(p,l)\le2\cdot3=6$. Then $15\le c_2=r_2^-(1,2)+r_2^-(2,1)+r(1,2)\le2r_2^-(p,3-p)+3\le2\cdot6+3=15$, hence $c_2=15$, $r_2^-(1,2)=r_2^-(2,1)=6$, $r_2(1)=r_2(2)=9$, $|K(p)|=3$ and $r(p,l)=3$ for each $l\in K(p)$.

Repeat the above reasoning with the pair $(p,l)$, $l\in K(p)\setminus\{3-p\}$, in the role of the pair $(1,2)$ to obtain $r_2(l)=9$ for both $l\in K(p)\setminus\{3-p\}\subseteq[3,6]$. Then, however,
\begin{equation*}
15=c_2=\frac12\sum_{l=1}^6r_2(l)\ge\frac12\left[r_2(1)+r_2(2)+\sum_{l\in K(p)\setminus\{3-p\}}r_2(l)\right]=\frac{4\cdot9}2=18,
\end{equation*}
a contradiction. $\ep$

\begin{claim}\label{r6}
If $i\in[1,6]$, then $r_2(i)\le6$.
\end{claim}
\bp Since $\Delta(G)\le3$, the claim is a direct consequence of Claim~\ref{r2}. $\ep$
\vskip1mm

\begin{claim}\label{K33}
The following statements are true:

$1.$ $\Delta(G)=3$;

$2.$ $G$ is a subgraph of $K_{3,3}$;

$3.$ $c_2\le18$.
\end{claim}

\bp 1. The assumption $\Delta(G)\le2$ would mean, by Claim~\ref{r2}, $r_2(i)\le2\cdot2=4$ for $i\in[1,6]$ and $c_2=\frac12\sum_{i=1}^6r_2(i)\le\lfloor\frac{24}2\rfloor=12$, a contradiction.

2. From Claims~\ref{K33}.1 and \ref{add}it follows that there is a partition $\{I,K\}$ of $[1,6]$ satisfying $|I|=|K|=3$ such that $r(i,k)\ge1$ with $\{i,k\}\in\binom{[1,6]}{2}$ implies $|\{i,k\}\cap I|=1=|\{i,k\}\cap K|$. Thus, $G$ is a subgraph of $K_{3,3}$ with the bipartition $\{I,K\}$.

3. Finally, by Claim~\ref{r2}, $c_2=\sum_{i\in I}\sum_{k\in K}r(i,k)\le9\cdot2=18$. $\ep$
\vskip1mm

Henceforth we suppose (w) that the bipartition of the graph $K_{3,3}$ from Claim~\ref{K33}.2 is $\{[1,3],[4,6]\}$, which leads to
\begin{equation*}
C_2=\bigcup_{i=1}^3\bigcup_{k=4}^6\ro(i,k).
\end{equation*}
Note that this assumption somehow restricts the meaning of (w) in the subsequent analysis, namely the bijection $\rho:[1,6]\to[1,6]$ in Proposition~\ref{wlog} should satisfy $\rho([1,3])\in\{[1,3],[4,6]\}$.

\begin{claim}\label{le1}
There is at most one pair $(i,k)\in[1,3]\times[4,6]$ with $r(i,k)=0$.
\end{claim}
\bp If $|\{(i,k)\in[1,3]\times[4,6]:r(i,k)=0\}|\ge2$, Claim~\ref{r2} yields $c_2\le7\cdot2=14$, a contradiction. $\ep$

\begin{claim}\label{le18}
The following statements are true:

$1.$ If $(i,j,k)\in\{(1,2,3),(4,5,6)\}$, then $q-9\le r(i,j,k)\le9$;

$2.$ $q\in[16,18]$.
\end{claim}
\bp 1. From Claim~\ref{le1} it follows that $\max(\deg_G(p):p\in[1,3])=3=\max(\deg_G(p):p\in[4,6])$. So, by Claim~\ref{D3}, $r(i,j,k)\ge q-9\ge7$. 

2. With $q\ge19$ we have $\max(r(1,2,3),r(4,5,6))\ge10$, which contradicts Claim~\ref{r3r3}. $\ep$
\vskip1mm

Use for an edge $\{i,k\}$ of the graph $K_{3,3}$ with bipartition $\{[1,3],[4,6]\}$  the label $r(i,k)\in[0,2]$. A colour $\alpha\in C_2$ \textit{corresponds} to a set $E\subseteq E(K_{3,3})$ if there is $\{i,k\}\in E$ such that $\alpha\in\ro(i,k)$. We denote by $\Col(E)$ the set of colours corresponding to $E$. Colours $\alpha,\beta\in C_2$, $\alpha\ne\beta$, are \textit{column-related} (in $M$) provided that $\ro(\alpha)\cap\ro(\beta)=\emptyset$; then
the pair $\{\alpha,\beta\}$ is not row-based, and hence it is column-based. Evidently, if $\gamma_j\in\ro(i_j,k_j)$ for $j\in[1,l]$ and $\{\gamma_j:j\in[1,l]\}$ is a set of pairwise column-related 2-colours, then $\{\{i_j,k_j\}:j\in[1,l]\}$ is a matching in $K_{3,3}$, and so $l\le3$.

A nonempty set $A\subseteq C_2$ is of the \textit{type} $n_1^{a_1}\negthinspace\dots n_k^{a_k}$ if $(n_1,\dots,n_k)$ is a decreasing sequence of integers from the interval $[1,|A|]$ such that $|\{j\in[1,q]:|A\cap\co(j)|=n_i\}|=a_i\ge1$ ($a_i$ columns of $M$ contain $n_i$ colours of $A$) for each $i\in[1,k]$, and $\sum_{i=1}^k a_in_i=2|A|$. Clearly, the type of $A$ is unique. In fact, only types of 3-element sets $A$ such that colours of $A$ are pairwise column-related will be important for us. In such a case the type $n_1^{a_1}\negthinspace\dots n_k^{a_k}$ satisfies a necessary condition $\sum_{i=1}^ka_i\binom{n_i}2\ge\binom{|A|}{2}=3$, where the involved sum represents the number of pairs of colours of $A$ in columns of $M$ belonging to $\Cov(A)$.

\begin{claim}\label{A}
If $\{i,j,k\}=\{1,2,3\}$, $\{l,m,n\}=\{4,5,6\}$, $\alpha\in\ro(i,l)$, $\beta\in\ro(j,m)$, $\gamma\in\ro(k,n)$, and the set $\{\alpha,\beta,\gamma\}$ is of the type $2^3$, then

$1.$ $c_2=15$;

$2.$ each colour of $C_2\setminus\{\alpha,\beta,\gamma\}$ occupies exactly one position in the set $[1,6]\times\Cov(\alpha,\beta,\gamma)$;

$3.$ $C_2=\bigcup_{s\in\Cov(\alpha,\beta,\gamma)}\co(s)$.
\end{claim}

\bp 1. A colour $\delta\in\{\alpha,\beta,\gamma\}$ must be in the set $[1,6]\times\Cov(\{\alpha,\beta,\gamma\}\setminus\{\delta\})$, otherwise neither of pairs $\{\delta,\varepsilon\}$ with $\varepsilon\in\{\alpha,\beta,\gamma\}\setminus\{\delta\}$ is good in $M$ (because each such pair is necessarily column-based). Therefore, $15\le c_2=3+|C_2\setminus\{\alpha,\beta,\gamma\}|\le3+(6\cdot3-3\cdot2)=15$, $c_2=15$ and $|C_2\setminus\{\alpha,\beta,\gamma\}|=12$.

2. By Claim~\ref{A}.1, $|C_2\setminus\{\alpha,\beta,\gamma\}|=12$. Each colour of $C_2\setminus\{\alpha,\beta,\gamma\}$ occupies a position in $[1,6]\times\Cov(\alpha,\beta,\gamma)$, hence the number of positions in $[1,6]\times\Cov(\alpha,\beta,\gamma)$, occupied by colours of $C_2\setminus\{\alpha,\beta,\gamma\}$, is at least twelve. On the other hand, the number of positions in $[1,6]\times\Cov(\alpha,\beta,\gamma)$, that are not occupied by colours of $C_2\setminus\{\alpha,\beta,\gamma\}$, is twelve. Consequently, each colour of $C_2\setminus\{\alpha,\beta,\gamma\}$ appears in $[1,6]\times\Cov(\alpha,\beta,\gamma)$ exactly once.

3. The set $\bigcup_{s\in\Cov(\alpha,\beta,\gamma)}\co(s)$ clearly contains each colour of $\{\alpha,\beta,\gamma\}$ as well as each colour of $C_2\setminus\{\alpha,\beta,\gamma\}$ (by Claim~\ref{A}.2). The inclusion $\bigcup_{s\in\Cov(\alpha,\beta,\gamma)}\co(s)\subseteq C_2$ is trivial. $\ep$

\begin{claim}\label{B}
	If $\{i,j,k\}=\{1,2,3\}$, $\{l,m,n\}=\{4,5,6\}$, $\alpha\in\ro(i,l)$, $\beta\in\ro(j,m)$ and $\gamma\in\ro(k,n)$, then the set $\{\alpha,\beta,\gamma\}$ is of the type $3^11^3$ or $2^3$.
\end{claim}

\bp Possible types of the set $\{\alpha,\beta,\gamma\}$ (that satisfy the necessary condition) are $3^2$, $3^12^11^1$, $3^11^3$ and $2^3$.

If $\{\alpha,\beta,\gamma\}$ is of the type $3^2$, consider a colour $\delta\in C_2\setminus\{\alpha,\beta,\gamma\}$ not appearing in $[1,6]\times\Cov(\alpha,\beta,\gamma)$; the number of such $\delta$'s is at least $c_2-3-2(6-3)\ge6$.
Then the number of pairs $\{\delta,\varepsilon\}$ with $\varepsilon\in\{\alpha,\beta,\gamma\}$, that are good in $M$ (and necessarily row-based), is at most two, while $|\{\alpha,\beta,\gamma\}|=3$, a contradiction.

If $\{\alpha,\beta,\gamma\}$ is of the type $3^12^11^1$, reasoning similarly as in the proof of Claim~\ref{A} one can show that $c_2=15$, and each colour $\delta\in C_2\setminus\{\alpha,\beta,\gamma\}$ occupies exactly one position in $[1,6]\times\Cov(\alpha,\beta,\gamma)$. Observe that if $b\in\Cov(\alpha,\beta,\gamma)$ satisfies $\{\alpha,\beta,\gamma\}\cap\co(b)=\{\varepsilon\}$, there is (a unique) $a\in[1,6]$ such that $\{\alpha,\beta,\gamma\}\cap\ro(a)=\{\varepsilon\}$ and $\zeta=(M)_{a,b}\ne\varepsilon$. The second copy of $\zeta$ is in $[1,6]\times([1,q]\setminus\Cov(\alpha,\beta,\gamma))$; so, the number of pairs $\{\zeta,\eta\}$ with $\eta\in\{\alpha,\beta,\gamma\}\setminus\{\varepsilon\}$, that are good in $M$ (and necessarily row-based), is one, while $|\{\alpha,\beta,\gamma\}\setminus\{\varepsilon\}|=2$, a contradiction. $\ep$

\begin{claim}\label{C}
	If $\{i,j,k\}=[1,3]$, $\{l,m,n\}=[4,6]$, $\ro(i,l)=\{\alpha_1,\alpha_2\}$, $\ro(j,m)=\{\beta_1,\beta_2\}$, $\gamma_1\in\ro(k,n)$ and $a,b\in[1,2]$, then the set $\{\alpha_a,\beta_b,\gamma_1\}$ is of the type $3^11^3$.
\end{claim}

\bp If the claim is false, then, by Claims~\ref{B} and \ref{A}.2, (w) $\{\alpha_1,\beta_1,\gamma_1\}$ is of the type $2^3$, $\Cov(\alpha_1,\beta_1,\gamma_1)=[1,3]$, and $\alpha_2$ occupies exactly one position in $[1,6]\times[1,3]$. Clearly, $\alpha_2$ appears in the column of $M$ containing both $\beta_1$ and $\gamma_1$ (the pair $\{\beta_1,\gamma_1\}$ is column-based), for otherwise $\{\alpha_2,\beta_1,\gamma_1\}$ would be of the type $2^21^2$, which is impossible by Claim~\ref{B}; so, by the same claim, $\{\alpha_2,\beta_1,\gamma_1\}$ is of the type $3^11^3$, (w) $\Cov(\alpha_2,\beta_1,\gamma_1)=[1,4]$.

Proceeding similarly as above we see that $\beta_2$ appears in the column containing $\alpha_1$, $\gamma_1$, and $\{\alpha_1,\beta_2,\gamma_1\}$ is of the type $3^11^3$, so that the pair $\{\alpha_2,\beta_2\}$ can be good in $M$ only if $\{\alpha_2,\beta_2\}\subseteq\co(4)$. Consequently, $\{\alpha_2,\beta_2,\gamma_1\}$ is of the type $2^3$. If $\{\alpha_1,\beta_1\}\subseteq\co(b)$, $b\in[1,3]$, then $\Cov(\alpha_2,\beta_2,\gamma_1)=[1,4]\setminus\{b\}$, so that, by Claims~\ref{A}.2 and \ref{A}.3, $\co(b),\co(4)\subseteq C_2$ and $\co(b,4)=\co(b)\setminus\{\alpha_1,\beta_1\}=\co(4)\setminus\{\alpha_2,\beta_2\}$. In such a case any colour $\delta\in\co(b,4)\subseteq C_2$ satisfies $2=\exc(\delta)\ge|\co(b,4)\setminus\{\delta\}|=3$, a contradiction. $\ep$

\begin{claim}\label{D}
If $\{i,j,k\}=\{1,2,3\}$, $\{l,m,n\}=\{4,5,6\}$, $\ro(i,l)=\{\alpha_1,\alpha_2\}$, $\ro(j,m)=\{\beta_1,\beta_2\}$, $\ro(k,n)\in\{\{\gamma_1\},\{\gamma_1,\gamma_2\}\}$ and $C _2^1=\ro(i,l)\cup\ro(j,m)\cup\ro(k,n)$, then there is $a\in[1,q]$ such that $C_2^1\subseteq\co(a)$.
\end{claim}

\bp By Claim~\ref{C} we know that (among others) all of the following sets are of the type $3^11^3$: $\{\alpha_1,\beta_1,\gamma_1\}$, $\{\alpha_2,\beta_1,\gamma_1\}$, $\{\alpha_1,\beta_2,\gamma_1\}$ and $\{\alpha_1,\beta_1,\gamma_2\}$ (provided that $\gamma_2\in\ro(k,n)$). Then there is $a\in[1,q]$ with $\{\alpha_1,\beta_1,\gamma_1\}\subseteq\co(a)$. Now $|\{\alpha_2,\beta_1,\gamma_1\}\cap\co(a)|\ge2>1$, hence $|\{\alpha_2,\beta_1,\gamma_1\}\cap\co(a)|=3$ and $\alpha_2\in\co(a)$. A similar reasoning shows that $\beta_2\in\co(a)$ as well as $\gamma_2\in\co(a)$ (under the assumption $\gamma_2\in\ro(k,n))$. $\ep$
\vskip1mm

\begin{claim}\label{comp}
	If $\cM^1$ is a perfect matching in $K_{3,3}$, there are uniquely determined perfect matchings $\cM^2$ and $\cM^3$ in $K_{3,3}$ such that $\{\cM^1,\cM^2,\cM^3\}$ is a partition of $E(K_{3,3})$.
\end{claim}
\bp The set $E(K_{3,3})\setminus\cM^1$ induces a 6-vertex cycle in $K_{3,3}$ whose edge set has a unique partition $\{\cM^2,\cM^3\}$ into perfect matchings of $K_{3,3}$. $\ep$
\vskip1mm

For a matching $\cM$ in $K_{3,3}$ we denote by $\wt(\cM)$ the \textit{weight} of $\cM$, i.e., the sum of labels of edges of $\cM$.

\begin{claim}\label{E}
No perfect matching of $K_{3,3}$ is of weight $6$.
\end{claim}

\bp Suppose that $\wt(\cM)=6$, where $\cM=\{\{i,l\},\{j,m\},\{k,n\}\}$, $\{i,j,k\}=\{1,2,3\}$ and $\{l,m,n\}=\{4,5,6\}$. By Claim~\ref{r2} then $6=\wt(\cM)=r(i,l)+r(j,m)+r(k,n)\le2+2+2=6$, hence $r(i,l)=r(j,m)=r(k,n)=2$. Let $\ro(i,l)=\{\alpha_1,\alpha_2\}$, $\ro(j,m)=\{\beta_1,\beta_2\}$, $\ro(k,n)=\{\gamma_1,\gamma_2\}$, $C_2^1=\{\alpha_1,\alpha_2,\beta_1,\beta_2,\gamma_1,\gamma_2\}$ and $C_2'=C_2\setminus C_2^1$. By Claim~\ref{D}, there is $a\in[1,q]$ such that $C_2^1\subseteq\co(a)$, which implies (since $|C_2^1)|=6=|\co(a)|$) $C_2^1=\co(a)$.

If $s,t,u\in\{1,2\}$, then, by Claim~\ref{C}, the set $\{\alpha_s,\beta_t,\gamma_u\}$ is of the type $3^11^3$. That is why, if $b\in[1,q]\setminus\{a\}$, then $|\co(a)\cap\co(b)|\le2$; moreover, $|\co(a)\cap\co(b)|=2$ is possible only if $\co(a)\cap\co(b)\in\{\{\alpha_1,\alpha_2\},\{\beta_1,\beta_2\},\{\gamma_1,\gamma_2\}\}$.

Consider a colour $\delta\in C_2'$. From above it is clear that each of both copies of $\delta$ provides  at most three pairs $\{\delta,\varepsilon\}$, $\varepsilon\in C_2^1$, that are good in $M$ (one of them is row-based, while at most two are column-based). Therefore, all six pairs $\{\delta,\varepsilon\}$ with $\varepsilon\in C_2^1$ can be good in $M$ only if each copy of $\delta$ appears in a column containing both colours $\zeta_1,\zeta_2$ for a suitable $\zeta\in\{\alpha,\beta,\gamma\}$. In such a case, however, the number of colours $\delta\in C_2'$, fulfilling the condition that all pairs $\{\delta,\varepsilon\}$ with $\varepsilon\in C_2^1$ are good in $M$, is at most $\lfloor\frac12(6\cdot3-|C_2^1|)\rfloor=6<9=|C_2'|$, a contradiction; note that there are at most three $b$'s with $b\in[1,q]\setminus\{a\}$ satisfying $C_2^1\cap\co(b)\in\{\{\alpha_1,\alpha_2\},\{\beta_1,\beta_2\},\{\gamma_1,\gamma_2\}\}$. $\ep$

\begin{claim}\label{sp}
The following statements are true:

$1.$ $c_2=15$;

$2.$ $c_{4+}=0$;

$3.$ there are $I,K\in\{[1,3],[4,6]\}$, $I\ne K$, and $k\in K$ such that for any $i\in I$ and any $l\in K\setminus\{k\}$ it holds $r(i,k)=1$ and $r(i,l)=2$.
\end{claim}

\bp 1., 2. Given a perfect matching $\cM^1$ of $K_{3,3}$, by Claim~\ref{comp} we know that there is a unique partition $\{\cM^1,\cM^2,\cM^3\}$ of $E(K_{3,3})$ into perfect matchings of $K_{3,3}$. By Claims~\ref{+m}.3 and \ref{E} then $15\le c_2=\sum_{n=1}^3\wt(\cM^n)\le\sum_{n=1}^35=15$, $c_2=15$, $c_{4+}=0$ (Claim~\ref{+m}.8) and $\wt(\cM^n)=5$, $n=1,2,3$; thus $\wt(\cM)=5$ for each perfect matching $\cM$ of $K_{3,3}$ ($\cM$ can be chosen as $\cM^1$). Among other things this means that no edge of $K_{3,3}$ is labelled with 0: otherwise any perfect matching of $K_{3,3}$ containing such an edge would be of weight at most $2\cdot2=4$ (Claim~\ref{r2}), a contradiction. 

3. Denote by $l(e)$ the label of an edge $e\in E(K_{3,3})$, and by $l_n$ the number of edges of $K_{3,3}$ labelled with $n$, $n=1,2$; then $l_1+l_2=9$, $15=c_2=l_1+2l_2=9+l_2$, $l_2=6$ and $l_1=3$. Let $\{e_1,e_2,e_3\}=\{e\in E(K_{3,3}):l(e)=1\}$.

If $a,b\in[1,3]$, $a\ne b$, then $e_a\cap e_b\ne\emptyset$. To see it suppose that $e_a\cap e_b=\emptyset$, and take $e\in E(K_{3,3})\setminus\{e_a,e_b\}$ such that $\{e_a,e_b,e\}$ is a perfect matching of $K_{3,3}$. The 6-vertex cycle in $K_{3,3}$ with the edge set $E(K_{3,3})\setminus\{e_a,e_b,e\}$ has at least five edges labelled with 2, hence one can find in $K_{3,3}$ a perfect matching $\cM\subseteq E(K_{3,3})\setminus\{e_a,e_b,e\}$ with $\wt(\cM)=3\cdot2=6$, a contradiction. 

Thus $e_1\cap e_2\ne\emptyset$, $e_1\cap e_3\ne\emptyset$ and $e_2\cap e_3\ne\emptyset$. Since the subgraph of $K_{3,3}$ induced by the set of edges $\{e_1,e_2,e_3\}$ is bipartite, this is possible only if there is a vertex $k\in[1,6]=V(K_{3,3})$ such that $e_1\cap e_2\cap e_3=\{k\}$. Having in mind that the bipartition of $K_{3,3}$ is $\{[1,3],[4,6]\}$, there are $I,K\in\{[1,3],[4,6]\}$ such that $I\ne K$, $k\in K$, and for any $i\in I$ and any $l\in K\setminus\{k\}$ it holds $r(i,k)=1$ and $r(i,l)=2$. $\ep$
\vskip1mm

Based on Claim~\ref{sp}.3 we suppose (w) $I=[1,3]$, $K=[4,6]$ and $k=6$ so that for any $i\in[1,3]$ and any $l\in[4,5]$ we have $r(i,6)=1$ and $r(i,l)=2$. Let $\ro(i,6)=\{\alpha_{i,6}\}$, $i=1,2,3$. 

If $\cM$ is a perfect matching in $K_{3,3}$, there is $s\in[1,6]$ such that $\cM=\cM^s$, where 
\begin{alignat*}{2}
\cM^1&=\{\{1,6\},\{2,4\},\{3,5\}\},\quad &\cM^2&=\{\{1,5\},\{2,6\},\{3,4\}\},\\
\cM^3&=\{\{1,4\},\{2,5\},\{3,6\}\}, &\cM^4&=\{\{1,6\},\{2,5\},\{3,4\}\},\\
\cM^5&=\{\{1,4\},\{2,6\},\{3,5\}\}, &\cM^6&=\{\{1,5\},\{2,4\},\{3,6\}\}. 
\end{alignat*}
Applying Claim~\ref{D} on five colours of $\Col(\cM^s)$ we see that there is $a^s\in[1,q]$ such that $\Col(\cM^s)\subseteq\co(a^s)$. Since $|\Col(\cM^s)\cap\Col(\cM^t)|\le2$ for $s,t\in[1,6]$, $s\ne t$, it is clear that $a^s\ne a^t$. From now on (w) 
\begin{align*}
(M)_{i,i}&=\alpha_{i,6}=(M)_{6,i+3},\\
\{\co(i)\cap C_2,\co(i+3)\cap C_2\}&=\{\Col(\cM^i),\Col(\cM^{i+3})\},\ i=1,2,3.
\end{align*}
As a consequence of Claims~\ref{sp}.1 and \ref{sp}.2 then all positions in the set 
\[
S=\{(1,4),(2,5),(3,6),(6,1),(6,2),(6,3)\}
\]
are occupied by 3-colours, and the same is true for the set of positions $[1,6]\times[7,q]$. Let $C_3^*=C_3\setminus(\ro(1,2,3)\cup\ro(4,5,6))$ and $c_3^*=|C_3^*|$. 

\begin{claim}\label{H}
Each position in the set $S$ is occupied by a colour of $C_3^*$.
\end{claim} 

\bp If a position $(i,i+3)$ with $i\in[1,3]$ is occupied by a colour $\beta\in\ro(1,2,3)$, that copy of $\beta$ provides no pair $\{\beta,\gamma\}$ with $\gamma\in\ro(4,5,6)$ that is good in $M$. Claim~\ref{le18} yields $\min(r(1,2,3),r(4,5,6))\ge q-9\ge7$. However, the number of pairs $\{\beta,\gamma\}$ with $\gamma\in\ro(4,5,6)$, that are good in $M$ (and necessarily column-based), is at most $\sum_{l\in\Cov(\beta)\setminus\{i+3\}}|\ro(4,5,6)\cap\co(l)|\le2\cdot3=6<r(4,5,6)$, a contradiction.

Similarly, if a position $(6,j)$ with $j\in[1,3]$ is occupied by a colour $\delta\in\ro(4,5,6)$, then the number of pairs $\{\delta,\varepsilon\}$ with $\varepsilon\in\ro(1,2,3)$, that are good in $M$, is at most $\sum_{l\in\Cov(\delta)\setminus\{j\}}|\ro(1,2,3)\cap\co(l)|\le2\cdot3=6<r(1,2,3)$, a contradiction. $\ep$

\begin{claim}\label{I}
$C_3^*\subseteq\ro(6)$.
\end{claim}

\bp Consider a colour $\beta\in C_3^*$, and let by $n_i$, $i\in\ro(\beta)$, denote the number of pairs $\{\beta,\gamma\}$ with $\gamma\in\{\alpha_{1,6},\alpha_{2,6},\alpha_{3,6}\}$, that are good in $M$, and are provided by the copy of $\beta$ in the row $i$ of $M$. If $i\in[1,3]$, then $n_i=1$ (with $\gamma=\alpha_{i,6}$), while $i\in[4,5]$ implies $n_i=0$, and $i=6$ yields $n_i=n_6=3$. Now, provided that $\beta\notin\ro(6)$, from $1\le|\ro(\beta)\cap[1,3]|\le2$ we obtain $\sum_{i\in\ro(\beta)}n_i=|\ro(\beta)\cap[1,3]|\le2<|\{\alpha_{1,6},\alpha_{2,6},\alpha_{3,6}\}|$, a contradiction. $\ep$

\begin{claim}\label{J}
$q=16$, and there is a $3$-colour $\beta\in\ro(i,m,6)$ with $i\in[1,3]$ and $m\in[4,5]$ that occupies a position in $\{6\}\times[7,16]$.
\end{claim}

\bp Since $c_{4+}=0$ (Claim~\ref{sp}.1) and $r_2(6)=\sum_{l=1}^3r(l,6)=3$, by Claims~\ref{I} and \ref{le18}.1 we have $q=c_2(6)+c_3(6)=3+[r(4,5,6)+c_3^*]$ and $c_3^*=q-3-r(4,5,6)\ge q-3-(q-9)=6$. On the other hand, from Claim~\ref{sp}.1 we get $|C|=2q+5=c_2+c_3=15+[r(1,2,3)+r(4,5,6)+c_3^*]$ so that $q-3=r(4,5,6)+c_3^*=2q-10-r(1,2,3)$, $r(1,2,3)=q-7$, and then Claim~\ref{le18}.1 yields $9\ge r(1,2,3)=q-7\ge16-7=9$, $r(1,2,3)=9$ and $q=16$. Using Claim~\ref{le18}.1 again we obtain $7=q-9\le r(4,5,6)=q-3-c_3^*=13-c_3^*\le7$, $r(4,5,6)=7$ and $c_3^*=6$.

The number of positions in $[1,3]\times[1,16]$ occupied by colours of $C_3^*$ is equal to $3\cdot16-c_2-3r(1,2,3)=48-15-27=6=c_3^*$, and each colour $\gamma\in C_3^*$ is involved in that counting, since $1\le|\ro(\gamma)\cap[1,3]|\le2$. Therefore, for any $\gamma\in C_3^*$ we get $|\ro(\gamma)\cap[1,3]|=1$ and $|\ro(\gamma)\cap[4,6]|=2$. 

Let $\beta\in C_3^*$ occupy a position in $\{6\}\times[7,16]$; the number of such $\beta$'s is $c_3^*-3=3$, because $C_3^*\subseteq\ro(6)$ (Claim~\ref{I}), the positions in $\{6\}\times[1,3]$ are occupied by colours of $C_3^*$ (Claim~\ref{H}) and $(M)_{6,l}=\alpha_{l-3,6}\in C_2$, $l=4,5,6$. Then $\ro(\beta)=\{i,m,6\}$, where $i\in[1,3]$ and $m\in[4,5]$. $\ep$
\vskip1mm

We are now ready to finish our analysis by showing that for a colour $\beta\in C_3^*$ of Claim~\ref{J} the number of pairs $\{\beta,\gamma\}$ with $\gamma\in C_2$, that are good in $M$, is less than $c_2=15$, which represents a final contradiction proving Theorem~\ref{16+}. 

First of all, if $\beta$ occupies a position in $\{i\}\times[7,16]$, then all pairs $\{\beta,\gamma\}$ with $\gamma\in C_2$, that are good in $M$, are row-based. In such a case the number of such pairs is $r_2(i)+r_2(m)+r_2(6)-[r(i,m)+r(i,6)]=5+6+3-(2+1)=11<15$, a contradiction.

Therefore, we have $\beta=(M)_{i,i+3}$, and this is the only copy of $\beta$ able to provide a column-based pair $\{\beta,\gamma\}$ that is good in $M$. Now we can find explicitly a colour $\gamma\in C_2$ such that the pair $\{\beta,\gamma\}$ is not good in $M$. Indeed, then $\beta\in\co(i+3)\cap C_2=\Col(\cM)$, where the perfect matching $\cM$ in $K_{3,3}$ satisfies $\cM=\{\{i,6\},\{j,m\},\{k,n\}\}$, $\{i,j,k\}=\{1,2,3\}$, $m\in[4,5]$ and $n=9-m$. Then $\co(i+3)=\{\beta\}\cup\ro(i,6)\cup\ro(j,m)\cup\ro(k,n)$, and so $\ro(j,n)\cap\co(i+3)=\emptyset$ (recall that $r(i,6)=1$ and $r(j,m)=2=r(k,n)$); thus, the pair $\{\beta,\gamma\}$ with $\gamma\in\ro(j,n)$ is not column-based.  Moreover, $\ro(\beta)\cap\ro(\gamma)=\{i,m,6\}\cap\{j,n\}=\emptyset$, the pair $\{\beta,\gamma\}$ is not row-based, hence it is not good in $M$, a contradiction.
\end{proof}
\vskip1mm

The solution of the problem of determining $\achr(K_6\square K_q)$ is now complete. It is summarised in the final theorem of the paper, where 
\begin{align*}
J_3&=[2,3]\cup\{q\in[41,\infty):q\equiv1\hskip-4mm\pmod2\},\\
J_4&=\{1,4,7\}\cup[16,40]\cup\{q\in[42,\infty):q\equiv0\hskip-3.9mm\pmod2\},\\
J_5&=\{5,8\},\\
J_6&=\{6\}\cup[9,15].
\end{align*}

\begin{theorem}\label{final}
If $a\in[3,6]$ and $q\in J_a$, then $\achr(K_6\square K_q)=2q+a$.
\end{theorem}

\begin{proof}
The achromatic number of $K_6\square K_q$ was analysed in \cite{HoPu} for $q\le4$ (for $q\le3$ see also Chiang and Fu~\cite{ChiF}), in Hor\v n\'ak and P\v cola~\cite{HoPc} for $q=5$, in~\cite{B} for $q=6$, in~\cite{Ho2} for $q=7$, and in~\cite{Ho1} for $q\in[41,\infty)$ with $q\equiv1\pmod2$. The remaining statements have been proved in the present paper, see Theorem~\ref{K6K8} for $q=8$, Theorem~\ref{915f} for $q\in[9,15]$, and Theorem~\ref{16+} for $q$ satisfying either $q\in[16,40]$ or $q\in[42,\infty)$ together with $q\equiv0\pmod2$.
\end{proof}

\begin{corollary}
If $q\in[1,\infty)$, then $2q+3\le\achr(K_6\square K_q)\le2q+6$. $\qed$
\end{corollary}
\vskip1mm

\noindent{\large{\bf Acknowledgements.}} This work was supported by the Slovak Research and Development Agency under the contract APVV-19-0153.


\begin{thebibliography}{9}

\bibitem{B} A. Bouchet, Indice achromatique des graphes multiparti complets et r\'eguliers, Cahiers du Centre d'\'Etudes de Recherche Op\'erationelle 20 (1978) 331--340.

\bibitem{ChiF} N.P. Chiang, H.L. Fu, On the achromatic number of the cartesian product $G_1\times G_2$, Australas. J. Combin. 6 (1992) 111--117.

\bibitem{HaHePr} F. Harary, S. Hedetniemi, G. Prins, An interpolation theorem for graphical homomorphisms, Portug. Math. 26 (1967) 454--462.

\bibitem{Ho1} M. Hor\v n\'ak, The achromatic number of $K_6\square K_q$ equals $2q+3$ if $q\ge41$ is odd, accepted in Discuss. Math. Graph Theory, https://doi.org/ 10.7151/dmgt.2420.

\bibitem{Ho2} M. Hor\v n\'ak, The achromatic number of $K_6\square K_7$ is $18$, Opuscula Math. 41 (2021) 163--185.

\bibitem{HoPc} M. Hor\v n\'ak, \v S. P\v cola, Achromatic number of $K_5\times K_n$ for small $n$, Czechoslovak Math. J. 53 (2003) 963--988.

\bibitem{HoPu} M. Hor\v n\'ak, J. Puntig\'an, On the achromatic number of $K_m\times K_n$, in: M. Fiedler (Ed.), Graphs and Other Combinatorial Topics, Teubner, Leipzig, 1983, pp. 118--123.

\bibitem{IK} W. Imrich, S. Klav\v zar, Product Graphs, Wiley-Interscience, New York, 2000.
\end{thebibliography}
\end{document}